\documentclass{amsart}

\usepackage{float}
\usepackage{graphicx}

\usepackage{tikz}
\usepackage{tikz-cd}

\usetikzlibrary{snakes}
\usepackage[all]{xy}
\usepackage{framed}
\usepackage{mathrsfs}
\usepackage{amssymb, amsmath}
\usepackage{a4wide}

\theoremstyle{plain}
\newtheorem{theorem}{Theorem}[section]
\newtheorem{lemma}[theorem]{Lemma}
\newtheorem{corollary}[theorem]{Corollary}
\newtheorem{proposition}[theorem]{Proposition}

\theoremstyle{definition}
\newtheorem{definition}[theorem]{Definition}
\newtheorem{example}[theorem]{Example}

\theoremstyle{remark}
\newtheorem{remark}[theorem]{Remark}

\numberwithin{equation}{section}

\newcommand{\C}{\mathbb{C}}

\newcommand{\R}{\mathbb{R}}
\newcommand{\Z}{\mathbb{Z}}

\newcommand{\mc}{\mathcal}

\usepackage{color}

\begin{document}

\title{Torus orbifolds with two fixed points}

\author{Alastair Darby}
\address{Department of Mathematical Sciences, Xi'an Jiaotong-Liverpool University, 111 Ren'ai Road,
Suzhou, 215123, Jiangsu Province, China.}
\curraddr{}
\email{Alastair.Darby@xjtlu.edu.cn}

\author{Shintaro Kuroki}
\address{Department of Applied Mathematics, Okayama University of Science, 1-1 Ridai-cho Kita-ku Okayama-shi Okayama 700-0005, Japan}
\curraddr{}
\email{kuroki@xmath.ous.ac.jp}

\author{Jongbaek Song}
\address{Department of Mathematical Sciences, KAIST, Daehak-ro 291, Daejeon 34141, Republic of Korea}
\curraddr{}
\email{jongbaek.song@gmail.com}

\thanks{The second author was supported by JSPS KAKENHI Grant Number 17K14196. 
The third author has been supported by the POSCO Science Fellowship of POSCO TJ Park Foundation and Basic Science Research Program through the National Research Foundation of Korea (NRF) funded by the Ministry of Education (NRF-2018R1D1A1B07048480).}



\maketitle

\begin{abstract}
The main objects of this paper are torus orbifolds that have exactly two fixed points. We study the equivariant topological type of these orbifolds and consider when we can use the results of~\cite{DKS} to compute its integral equivariant cohomology, in terms of generators and relations, coming from the corresponding orbifold torus graph.
\end{abstract}

\section{Introduction}
The relationship between the equivariant cohomology of certain smooth manifolds $M^{2n}$, with half-dimensional torus actions $\mathbb{T}^n=S^1\times\dots\times S^1$, and the combinatorics of their quotient spaces is well-known. A toric manifold, defined as a compact non-singular toric variety, admits a locally standard action of a half-dimensional torus. This implies that its quotient space is a manifold with faces. The equivariant cohomology is then simply given as the \emph{face-ring} of this quotient space. This can also be realized as the face-ring of the complete regular fan associated to the toric variety. A quasitoric manifold, first defined in~\cite{DJ}, is a topological generalization of a toric manifold whose quotient space is a simple polytope (an example of a manifold with faces) but which is more general than a smooth projective toric variety. Its equivariant cohomology is then given by the face-ring associated to the simple polytope.

Torus manifolds, appearing in~\cite{HM}, are a wider class of manifolds $M^{2n}$, with half-dimensional torus action, that contain both toric and quasitoric manifolds. The main example of a torus manifold that is neither a toric nor a quasitoric manifold is the even-dimensional sphere $S^{2n}\subset \C^n\oplus \R$, where the torus acts coordinatewise on the first $n$~coordinates. To any torus manifold we can associate a combinatorial object, called a \emph{torus graph} (see~\cite{MMP}), which is an $n$-valent graph whose vertices correspond to the fixed points of $M$ and whose edges are labelled by irreducible torus representations. The underlying graph for $S^{2n}$ is the one with exactly two vertices and $n$ edges between them. We can then calculate the \emph{graph equivariant cohomology} of the torus graph, which is a ring of \emph{piecewise polynomials} on the graph. When the ordinary cohomology of $M$ is trivial in all odd degrees, then the torus action is locally standard and the equivariant cohomology of $M$ is isomorphic to the graph cohomology of the associated torus graph. It is also possible to give explicit generators and relations for this ring via \emph{Thom classes} of the torus graph and we can see that this is isomorphic to the face-ring of the quotient space which is, in this case, a manifold with faces.

When we move from manifolds to orbifolds the picture slightly changes. In the case of singular toric varieties, including toric varieties having orbifold singularities,  it was proved in \cite{BFR2,F} that its equivariant cohomology with integer coefficients is given by the piecewise polynomials on its fan if its ordinary odd degree cohomology vanishes. For toric manifolds the ring of piecewise polynomials on the fan is isomorphic to the face-ring of its quotient but this is not true for orbifolds in general. In \cite{BSS} it is shown that the equivariant cohomology of projective toric orbifolds, under the condition of vanishing odd degree cohomology, can be realized as a subring of the usual face-ring of the fan that satisfies an \emph{intergrality condition}.

As a generalization of torus manifolds, Hattori--Masuda~\cite{HM} introduced a notion of \emph{torus orbifolds} which are $2n$-dimensional, closed and oriented orbifold 
with an effective $n$-dimensional torus $T^n$-action whose fixed point set is nonempty. In this paper, we mainly focus on the topology of a torus orbifold having exactly two 
isolated fixed points.

In Section~2 we consider the main combinatorial object of this paper which is a \emph{manifold with faces} that has exactly two vertices. The suspension of a simplex gives the simplest example of such an object and we show that, up to \emph{combinatorial equivalence}, these are the only ones. Torus orbifolds are discussed in Section~3 and we show how one can be constructed from a manifold with faces, with some additional information, using a \emph{quotient construction.}

We restrict our attention to torus orbifolds over the suspension of a simplex in Section~4  and discuss their equivariant topological type. More specifically we show that every such torus orbifold is equivariantly homeomorphic to the quotient of an even dimensional sphere by a product of cyclic groups. Interestingly, we also find a condition on the combinatorial data for when this torus orbifold is equivariantly homeomorphic to an even dimensional sphere.

In Section~5 we give the main result from~\cite{DKS} which computes the equivariant cohomology of a torus orbifold $X$ whose odd degree ordinary cohomology is trivial. To do this we associate to each torus orbifold over a manifold with faces a labelled graph, called a \emph{torus orbifold graph}, which generalizes the torus graphs of~\cite{MMP}. We can then describe the equivariant cohomology of $X$ as the ring of piecewise polynomials on the associated torus orbifold graph if $H^{\text{odd}}(X)=0$. We show that this is isomorphic to a \emph{weighted face ring} that gives us a description of $H^*_T(X)$ in terms of generators and relations. Restricting to torus orbifolds $X$ over the suspension of a simplex we compute these ring structures explicitly. 

\section{Manifold with two vertices}
\label{sect:2}

In this first section, we introduce the combinatorial object which we deal with in this paper. We refer to the book \cite[Section 7.1]{BP} for the notation used in this section.
Let $Q_{1}$, $Q_{2}$ be $n$-dimensional manifolds with faces.
If there is a homeomorphism $f:Q_{1}\to Q_{2}$ which preserves faces for two manifolds with faces $Q_{1}$ and $Q_{2}$, then we call $Q_{1}$ and $Q_{2}$ \textit{isomorphic} (in the sense of manifolds with faces).
We can also define a weaker equivalence relation called {\it combinatorial equivalence} among manifolds with faces as follows.
We may regard the set of faces of a manifold with faces $Q$ as a partially ordered set by the inclusions of faces, say $\mathcal{S}(Q)$.
If there is a bijective map between $\mathcal{S}(Q_{1})$ to $\mathcal{S}(Q_{2})$ which preserves the order, then we call two manifolds with faces $Q_{1}$ and $Q_{2}$ {\it combinatorially equivalent}.
It is easy to check that if $Q_{1}$ and $Q_{2}$ are isomorphic, then $Q_{1}$ and $Q_{2}$ are combinatorially equivalent.
However, the converse is not true, see Remark~\ref{connected-sum-construction}.
If a manifold with faces $Q$ has exactly two vertices (i.e., $0$-dimensional faces),
we call such $Q$ a \textit{manifold with two vertices} for short.

Let $Q$ be an $n$-dimensional manifold with two vertices.
By the definition of $Q$,  
the following properties hold: if $n>1$, 
\begin{enumerate}
\item there exist exactly $n$ facets, say $\mathfrak{F}(Q)=\{F_{1},\ldots, F_{n}\}$;
\item if $0<k<n$, the intersection $\bigcap_{j=1}^{k}F_{i_{j}}$ is connected, i.e., a codimension-$k$ face of $Q$;
\item conversely, for every codimensional-$k$ face $H$ ($k\not=0,n$), there exist exactly $k$ distinct facets $F_{i_{1}},\ldots, F_{i_{k}}$ such that $H=\bigcap_{j=1}^{k}F_{i_{j}}$;
\item $\bigcap_{j=1}^{n}F_{j}=\{p,q\}$ (the set of all vertices of $Q$).
\end{enumerate}
If $n=1$, $Q$ is nothing but the interval $[-1,1]$, i.e., there are only two facets $\{-1\}$ and $\{1\}$.

The typical example of an $n$-dimensional manifold with two vertices  for $n \ge 1$ is the suspension $\Sigma \Delta^{n-1}$ 
of the $(n-1)$-dimensional simplex $\Delta^{n-1}$ (see Figure~\ref{fig_susp_of_simp} for the case when $n=2$ and $n=3$).
More precisely, in this paper, $\Sigma \Delta^{n-1}$ is defined by
\begin{align*}
\Delta^{n-1}\times [-1,1]/\sim,
\end{align*}
where the equivalence relation $\sim$ is defined by collapsing $\Delta^{n-1}\times \{-1\}$ (resp. $\Delta^{n-1}\times \{1\}$) to the vertex $p$ (resp.~$q$).
Note that a codimension-$k$ face $H$ of $\Sigma \Delta^{n-1}$, for $k=0,\ldots, n-1$, is determined by 
the suspension $\Sigma F$ of some codimension-$k$ face $F$ in $\Delta^{n-1}$ and codimension-$n$ faces are the two vertices $p,\ q$ in $\Sigma \Delta^{n-1}$.
\begin{figure}
\label{suspension}
\begin{tikzpicture}
\begin{scope}[xscale=1.3]
\node[right] at (0,0) {$\Delta^1$};
\draw[dotted, thick] (0,1)--(0,-1);

\draw[fill=yellow, opacity=0.5] (0,0) circle [radius=1];
\node at (-1,0) {$\bullet$};
\node at (1,0) {$\bullet$};
\node[right] at (0,0) {$\Delta^1$};
\node at (0,-1.5) {$\Sigma \Delta^1$};

\node at (0,-1.5) {$\Sigma \Delta^1$};
\end{scope}

\begin{scope}[xshift=170]
\draw[dotted, thick, fill=red!50] (-0.3,0.94) -- (0.3, 0.47)--(0,-1)--cycle;
\node[left] at (-2,-1/2) {$\Delta^2$};
\draw[dashed, ->] (-2, -1/2)--(0,0);
\draw[fill=yellow, opacity=0.5] (-1.3,0) [out=75, in=180] to (0,1) [out=0, in=105] to (1.3,0)
	[out=360-105, in=0] to (0,-1) [out=180, in=-75] to (-1.3, 0) ;
\draw (-1.3,0) [out=75, in=180] to (0,1) [out=0, in=105] to (1.3,0)
	[out=360-105, in=0] to (0,-1) [out=180, in=-75] to (-1.3, 0) ;
\draw (-1.3,0) [out=45, in=180] to (0,1/2) [out=0, in=135] to (1.3,0);
\fill (-1.3,0) circle (0.07);
\fill (1.3,0) circle (0.07);
\node at (0,-1.5) {$\Sigma \Delta^2$};
\end{scope}
\end{tikzpicture}
\caption{Suspensions of simplices}\label{fig_susp_of_simp}
\end{figure}
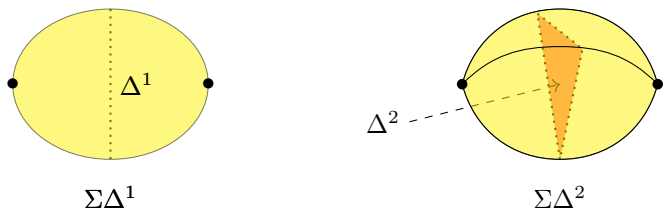

\begin{proposition}
\label{comb-str-mfdw2}
Let $Q$ be an $n$-dimensional manifold with two vertices. 
Then $Q$ is combinatorially equivalent to $\Sigma \Delta^{n-1}$. 
\end{proposition}
\begin{proof}
Let $\{F_{1},\ldots, F_{n}\}$ be the facets of $Q$.
Note that all of the lower dimensional facets are determined by the intersection of the $F_{i}$'s except the two vertices.
This combinatorial structure is exactly the same as that of $\Sigma \Delta^{n-1}$ discussed above.
\end{proof}

\begin{remark}
\label{connected-sum-construction}
We define the topology on $\Delta^{n-1}$ by the induced topology from $\mathbb{R}^{n}_{\ge}$, i.e., $\Delta^{n-1}:=\{(x_{1},\ldots,x_{n})\in \mathbb{R}_{\ge}^{n}\ |\ \sum_{i=1}^{n}x_{i}=1\}$.
By the connected sum with a homology sphere and $\Sigma \Delta^{n-1}$, we can construct  
an $n$-dimensional manifold with two vertices which is not {\it isomorphic} to $\Sigma \Delta^{n-1}$ (also see \cite{D-coxeter}).
\end{remark}


\section{Quotient construction of torus orbifolds}
\label{sect:3}
In this section, we briefly recall the quotient construction of a $2n$-dimensional torus orbifold. 
As an introduction to the subject we recommend \cite{ALR} for the basic facts regarding orbifolds and the papers \cite{GKRW, HM} to do with torus orbifolds in particular.
In this paper, a compact, connected, commutative, Lie group is called a {\it torus}, often denoted by $T$.
We use the following notation regarding a torus $T$:
\begin{itemize}
\item $\mathfrak{t}:={\rm Lie}(T)$, the Lie algebra of $T$;
\item $\mathfrak{t}_{\mathbb{Z}}:={\rm exp}^{-1}(e)\subset \mathfrak{t}$ is the lattice, where $e\in T$ is the identity element and ${\rm exp}:\mathfrak{t}\to T$ is the exponential map;
\item $\mathfrak{t}_{\mathbb{Q}}:=\mathfrak{t}_{\mathbb{Z}}\otimes_{\mathbb{Z}}\mathbb{Q}$.
\end{itemize}
We also use the following symbol for the {\it standard $n$-dimensional torus}:
\begin{align*}
\mathbb{T}^{n}:=\{(z_{1},\ldots, z_{n})\in \mathbb{C}^{n}\ |\ |z_{1}|=\cdots =|z_{n}|=1\}.
\end{align*}
It is well-known that every $n$-dimensional torus $T$ is isomorphic to $\mathbb{T}^{n}$.
So we note that the quotient by a finite subgroup $F$ of $\mathbb{T}^{n}$, i.e., $F$ is isomorphic to a product of cyclic groups, is always isomorphic to an $n$-dimensional torus.
\begin{remark}
When we consider a torus action on an orbifold, it is quite natural to distinguish between the standard torus $\mathbb{T}$ and $T=\mathbb{T}/F$, though they are isomorphic as Lie groups.
We shall discuss this again in Section 4.
\end{remark}

A torus orbifold defined in \cite{HM} is as follows:
\begin{definition}
A $2n$-dimensional closed oriented orbifold $X$ with an effective $n$-dimensional torus $T$-action is called a {\it torus orbifold} if it has a nonempty fixed point set $X^{T}$.
\end{definition}
A connected component of the fixed point set of a circle
subgroup $S_{i}$ of $T$ is a suborbifold, say $X_{i}$.
This suborbifold $X_{i}$ is called a {\it characteristic suborbifold} 
if $X_{i}$ is a $(2n-2)$-dimensional orbifold and contains at least one fixed point of the $T$-action. 
In the definition of torus orbifolds in \cite{HM}, 
we also need to choose an ``invariant normal orientation'' for every characteristic suborbifold $X_{i}$.
We will discuss this again for the case of torus orbifolds defined by the quotient construction.

Let $Q$ be an $n$-dimensional manifold with faces.
We write the set of facets of $Q$ as 
\begin{align*}
\mathfrak{F}(Q):=\{F_1, \dots, F_m\}.
\end{align*}
Let $T$ be an $n$-dimensional torus.
We identify the lattice $\mathfrak{t}_{\mathbb{Z}}$ with $\mathbb{Z}^{n}$.
A function 
\begin{align*}
\lambda \colon \mathfrak{F}(Q) \to \mathbb{Z}^{n}
\end{align*}
is called a \textit{characteristic function} if 
it satisfies the following condition:
\begin{align*}
\tag{$\ast$}
\text{  
$\{\lambda(F_{i_1}), \dots, \lambda(F_{i_k})\}$ is linearly independent 
whenever $F_{i_1}\cap \dots \cap F_{i_k} \neq \emptyset$, for $1\le k\le n$.} 
\end{align*}
We call $(Q,\lambda)$ an {\it orbifold characteristic pair}.

We note that for each point $x$ in $Q$, there is a unique face $F(x)$ of $Q$ containing $x$ in its relative interior. 
Condition $(\ast)$ allows us to define a torus 
\begin{align*}
T_{F(x)} := 
\left(
\mathbb{R}\lambda(F_{i_1})\times \cdots\times \mathbb{R}\lambda(F_{i_k})\right)/
\left(
\mathbb{Z}\lambda(F_{i_1})\times \cdots\times \mathbb{Z}\lambda(F_{i_k})\right)
\end{align*}
if $F(x)=(F_{i_1} \cap \dots \cap F_{i_k})^{o}$, where 
the symbol $F^{o}$ denotes a connected component of the set $F$.
Now we define the following quotient space: 
\begin{equation*}
\label{eq_Definition_of_torus_orb}
X(Q, \lambda):=(Q \times T) / \sim,
\end{equation*}
where the equivalence relation $\sim$ is given by 
\begin{equation*} 
(x,t) \sim (y,s) \text{ if and only if } x=y \text{ and } t^{-1}s \in T_{F(x)}.
\end{equation*}
Note that if $x\in Q$ is a vertex then $T_{F(x)}\cong T$.

It is easy to check that $X:=X(Q, \lambda)$ has the natural $T$-action induced from the multiplication of $T$ on the second factor of 
$Q\times T$, and 
its orbit projection $X\to X/T$ is induced from the 
projection onto the first factor $Q\times T\to Q$, i.e., $X/T$ is equipped with the structure of a manifold with faces by the induced homeomorphism $X/T\simeq Q$.
Note that if $F(q)$ (for $q\in Q$) is a codimension-$k$ face, then the inverse $\pi^{-1}(q)$ is an $(n-k)$-dimensional orbit which is homeomorphic to $T/T_{F(q)}$.
More precisely,
one can define an orbifold structure on  
$X(Q, \lambda)$ in a similar fashion to \cite[Section 2.1]{PS} (see \cite{DKS} for a more precise account).
Moreover, one can see that $\{\pi^{-1}(p)\mid \text{$p$ is a vertex of $Q$}\}$ is the set of fixed points.
Hence, we conclude that $X(Q, \lambda)$ is a torus orbifold.

Conversely, let $X$ be a $2n$-dimensional torus orbifold such that the orbit space $X/T$ is isomorphic to some $n$-dimensional manifold with faces $Q$.
The face structure on $Q$ is determined by inclusion relation for isotropy subgroups. Hence, for the orbit map $\pi:X\to X/T\simeq Q$, the preimage $X_{i}:=\pi^{-1}(F_{i})$ 
of a facet $F_{i}\in \mathfrak{F}(Q)$
is a characteristic suborbifold $X_{i}$ for every facet $F_{i}\in \mathfrak{F}(Q)$.
Let $S_{i}$ be the circle subgroup of $T$ fixing $X_{i}$.
Then, we can define the characteristic function $\lambda:\mathfrak{F}(Q)\to \mathbb{Z}^{n}$ by a choice of a nonzero vector $v_{i}\in \mathbb{Z}^{n}\cong \mathfrak{t}_{\mathbb{Z}}$ such that $S_{i}={\rm exp}\mathbb{R}v_{i}$.

Note that there are infinitely many choices of such nonzero vectors, i.e. if we take a primitive vector $s_{i}$ such that $S_{i}={\rm exp}\mathbb{R}s_{i}$, then for any $r\in \mathbb{Z}\setminus \{0\}$ the equality $S_{i}={\rm exp}\mathbb{R} (r s_{i})$ holds.
In order to determine the nonzero vector, we need to choose a ``normal orientation'' of $X_{i}$ in the following way.
Due to \cite[Lemma 12.1]{HM} (also see \cite[Proposition 2.12]{GKRW}), for every $x\in X_{i}$ there is a special chart $(U_{x}, V_{x}, H_{x}, p_{x})$ around $x$, also called a good local chart in \cite{ALR}, where $U_{x}$ is a neighborhood of $x$, $V_{x}$ is an open subset of $\mathbb{R}^{n}$ and $p_{x}: V_{x}\to U_{x}$ is an $H_{x}$-equivariant map which induces a homeomorphism $V_{x}/H_{x}\cong U_{x}$, where the finite group $H_{x}$ acts on $V_{x}$ as $H_{x}\subset O(V_{x})\simeq O(2n)$.
It satisfies the following properties:
\begin{enumerate}
\item the tangent space $T_{\widetilde{x}}V_{x}$ of $\widetilde{x}=p_{x}^{-1}(x)$ in $V_{x}$, identifying $V_{x}=T_{\widetilde{x}}V_{x}(\simeq \mathbb{R}^{2n})$, 
splits into $W_{ix}\oplus W_{ix}^{\perp}$, where 
$W_{ix}$ is tangent to $p_{x}^{-1}(U_{x}\cap X_{i})$ and $W_{ix}^{\perp}$ may be regarded as the normal vector space of $x\in X_{i}$;
\item $H_{x}$ acts on $W_{ix}^{\perp}$;
\item there is a (connected) finite cover $\widetilde{S}_{i}$ of $S_{i}$ and a lifting of the action of $S_{i}$ to the action of $\widetilde{S}_{i}$ on $V_{x}$
for any point $x\in X_{i}$;
\item the lifted action of $\widetilde{S}_{i}$ acts on the normal vector space $W_{ix}^{\perp}(\subset V_{x})$ non-trivially;
\item $\widetilde{S}_{i}$ acts on $V_{x}=W_{ix}\oplus W_{ix}^{\perp}$ effectively.
\end{enumerate}

We choose an orientation of $W_{ix}^{\perp}$ for every $x\in X_{i}$ which 
we call a {\it normal orientation} of $X_{i}$.
On the other hand, there are exactly two primitive vectors $s_{i}$ and $-s_{i}$ such that ${\rm exp}\mathbb{R}s_{i}={\rm exp}\mathbb{R}(-s_{i})=S_{i}$.
Note that the choice of a sign determines the orientation of $S_{i}$.
We take the orientation of $S_{i}$ such that lifted $\widetilde{S}_{i}$-action preserves the orientation of the given normal orientation of $X_{i}$.
Therefore, we can assign the primitive vector $s_{i}\in \mathfrak{t}_{\mathbb{Z}}$ to $X_{i}$ without sign ambiguity.
Moreover, the continuous map $p_{x}:V_{x}\to V_{x}/H_{x}\simeq U_{x}$ is an equivariant map with respect to the finite covering 
$\widetilde{S}_{i}\to \widetilde{S}_{i}/H_{x}\simeq S_{i}$, i.e., the $r_{i}$-times rotation map between circles for some $r_{i}\in \mathbb{N}$.
Thus we may write $\widetilde{S}_{i}={\rm exp}\mathbb{R}(r_{i}s_{i})$ and $S_{i}={\rm exp}\mathbb{R}s_{i}$ (also see equation \eqref{sequence}). 

A torus orbifold with a preferred normal orientation on each $X_{i}$ is called an {\it ``omnioriented'' torus orbifold}.
In this case, the characteristic function $\lambda:\mathfrak{F}(Q)\to \mathbb{Z}^{n}$ is uniquely determined
without ambiguity of scalar multiplications, i.e., $\lambda(F_{i})=v_{i}:=r_{i}s_{i}$.
We also have the following fact, see~\cite[Theorem 4.4]{DKS}:
\begin{lemma}
\label{fundamental-lem}
Given a torus orbifold $X$ with $X/T\simeq Q$ such that $H^{2}(Q;\mathfrak{t}_{\mathbb{Z}})=0$, let $\lambda$ be the characteristic function determined by the appropriate choice of omniorientation as above. Then, $X$ is equivariantly homeomorphic to $X(Q,\lambda)$.
\end{lemma}

\begin{remark}
In this paper, an equivariant homeomorphism means a weakly equivariant homeomorphism, i.e., an equivariant homeomorphism up to an automorphism of the groups acting on the spaces. We often denote them by $(X,T)\simeq (X',T')$.
\end{remark}

\begin{remark}
\label{meaning_of_choices}
Note that a choice of scalar multiplications of the $\lambda(F_{i})$'s 
changes the orbifold structure on $X$; however, it 
does not change the equivariant homeomorphism type of $X$.
Namely, if $\lambda(F_{i})=r_{i}\lambda'(F_{i})$ ($F_{i}\in \mathfrak{F}(Q)$) for some $r_{i}\in \mathbb{Z}\setminus \{0\}$, then
$X(Q, \lambda)\simeq X(Q, \lambda')$.
\end{remark}

\begin{example}[Spindle]
\label{spindle}
We denote the cyclic group of order $k$ by $C_{k}$ and consider the natural surjection $p_{k}:\mathbb{C}\to \mathbb{C}/C_{k}$.
We define a {\it spindle} $S^{2}(m,n)$, for $m,n\not=0$, as follows.
The underlying topological space of $S^{2}(m,n)$ is homeomorphic to the $2$-dimensional sphere $S^{2}$.
Denote the northpole of $S^2$ by $N$ and the southpole by $S$.
The orbifold structure of $S^{2}(m,n)$, for $m,n>0$, is 
the maximal orbifold atlas $\mathcal{U}=\{(U_{\alpha},V_{\alpha},H_{\alpha},p_{\alpha})\}$ which contains the following two orbifold charts:
(1) $(U_{S}, \mathbb{C}, C_{m}, p_{m})$ for $m>0$ around $S$ whose open neighborhood $U_{S}$ is defined by $S^{2}\setminus \{N\}$;
(2) $(U_{N}, \mathbb{C}, C_{n}, p_{n})$ for $n>0$ around $N$ whose open neighborhood $U_{N}$ is defined by $S^{2}\setminus \{S\}$.
The orbifold $S^{2}(m,m)$ is also called a {\it rugby ball} and $S^{2}(m, 1)$ or $S^{2}(1,n)$ is also called a {\it teardrop}.
If $m$ is negative, then we consider the orbifold chart on $U_{S}$ as $(U_{S}, \overline{\mathbb{C}}, C_{-m}, p_{-m})$, where 
$\overline{\mathbb{C}}$ is $\mathbb{C}$ with the reversed orientation, $C_{-m}$ acts on it by multiplication and $p_{-m}:\overline{\mathbb{C}}\to \overline{\mathbb{C}}/C_{-m}$ is the natural surjection. 
Similarly, we define the orbifold chart on $U_{N}$ when $n<0$.

Note that the standard $\mathbb{T}^{1}$-action on $\mathbb{C}$ induces $(U_{S}, \mathbb{T}^{1}/C_{m})\simeq (U_{S}, T^{1})$ and $(U_{N}, \mathbb{T}^{1}/C_{n})\simeq (U_{N}, T^{1})$ by
 $p_{m}$ and $p_{n}$ respectively. 
Moreover, because the underlying space $S^{2}$ has an effective $T^{1}$-action, 
$S^{2}(m,n)$ also has an effective $T^{1}$-action.
Then, there are two fixed point $\{N,S\}$.
Therefore, $(S^{2}(m,n),T^{1})$ is a torus orbifold 
and its orbit space $S^{2}(m,n)/T^{1}$ is the interval $[-1,1]$ such that $\{-1\}$ (resp.~$\{1\}$) corresponds to $S$ (resp.~$N$).
Then, $(U_{S}, \mathbb{C}, C_{m}, p_{m})$ (resp.~$(U_{S}, \overline{\mathbb{C}}, C_{-m}, p_{-m})$)
has the natural extension of the $T^1$-action whose normal orientation is determined 
by the standard orientation of $\mathbb{C}$ (resp.~$\overline{\mathbb{C}}$) and 
$p_{m}:\mathbb{C}\to U_{S}\simeq \mathbb{C}/C_{m}$ (resp.~$p_{-m}:\overline{\mathbb{C}}\to U_{S}\simeq \mathbb{C}/C_{-m}$) induces 
the homomorphism $\widetilde{T}^{1}\to T^{1}\simeq \widetilde{T}^{1}/C_{\pm m}$ by rotating $m$-times.
Similarly we have the $T^{1}$-extension on the orbifold chart on $U_{N}$.
Hence, the characteristic function $\lambda:\{-1,1\}\to \mathbb{Z}\setminus \{0\}$ is defined by $\lambda(-1)=m$ and $\lambda(1)=n$.
This means that the pair $([-1,1], \lambda)$ defines the spindle $S^{2}(m,n)$ (see Figure~\ref{spindle_fig}).

Note that for any $m,n\in\mathbb{Z}\setminus\{0\}$, $(S^{2}(m,n),T^{1})$ is equivariantly homeomorphic to $(S^{2},T^{1})\simeq (S^{2}(1,1),T^{1})$.
In particular, $(S^{2}(1, 1),T^{1})$, $(S^{2}(1, -1),T^{1})$, $(S^{2}(-1, 1),T^{1})$, $(S^{2}(-1, -1),T^{1})$ are equivariantly homeomorphic torus manifolds but have different omniorientations. 
Such different omniorientations give four invariant stably complex structures on $S^{2}$, also see Remark~\ref{meaning_of_choices}.
\begin{figure}[h]
\begin{tikzpicture}
\draw (0,0)--(3,0);
\node [above] at (0,0.1) {$m$};
\fill (0,0) circle [radius=2.5pt];
\node [above] at (3,0.1) {$n$};
\fill (3,0) circle [radius=2.5pt];
\end{tikzpicture}
\caption{The characteristic pair of the spindle $S^{2}(m,n)$.}
\label{spindle_fig}
\end{figure}
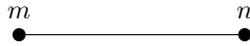
\end{example}

\section{Torus orbifolds over $\Sigma \Delta^{n-1}$}
\label{sect:4}

From this section, we focus on the case when $Q=\Sigma \Delta^{n-1}$.
If $n\ge 2$, a characteristic function $\lambda:\mathfrak{F}(Q):=\{F_{1},\ldots,F_{n}\}\to \mathbb{Z}^{n}$ is often illustrated by the following $(n\times n)$-square matrix in $GL(n;\mathbb{Q})\cap M_{n}(\mathbb{Z})$:
\begin{align}
\label{matrix}
\Lambda=\left(\lambda(F_{1})\mid \cdots \mid \lambda(F_{n}) \right)=
\begin{pmatrix}
\lambda_{11} & \cdots & \lambda_{1n} \\
\vdots & \ddots & \vdots \\
\lambda_{n1} & \cdots & \lambda_{nn}
\end{pmatrix}
\end{align}
where $\lambda(F_{i})$, $i=1,\ldots, n$, is a nonzero vector in $\mathbb{Z}^{n}$.
For simplicity, we shall write $X(\Lambda):=X(\Sigma\Delta^{n-1},\lambda)$.

\begin{remark}
If one takes $Q$ as a manifold with two vertices which is not isomorphic to $\Sigma \Delta^{n-1}$ (see Remark~\ref{connected-sum-construction}), then $X(Q,\lambda)$ may not be homeomorphic to $X(\Lambda)$.
In general, if such a $Q$ is not face acyclic, then a torus orbifold over $Q$ may not be recovered from the quotient construction discussed in Section \ref{sect:3}. 
\end{remark}

The goal of this section is to show Theorem~\ref{top-type}
which illustrates the 
equivariant topological type of a torus orbifold $X(\Lambda)$.
Let $\pi:X(\Lambda)\to \Sigma\Delta^{n-1}$ be the orbit projection of the $T^{n}$-action on $X(\Lambda)$ 
and $X_{i}:=\pi^{-1}(F_{i})$ the characteristic suborbifold associated to $F_{i}$.

Note that the integer square matrix $\Lambda$ defined in equation  \eqref{matrix} induces an isomorphism
$\Lambda:\mathbb{R}^{n}\to \mathbb{R}^{n}$; therefore, this induces an isomorphism
$\widetilde{\Lambda}:\mathbb{R}^{n}/\mathbb{Z}^{n}\to \mathbb{R}^{n}/\Lambda(\mathbb{Z}^{n})$.
Hence, we can define the following surjective homomorphism:
\begin{align}
\label{sequence}
\xymatrix{
\mathbb{T}^{n}=\mathbb{R}^{n}/\mathbb{Z}^{n} \ar[r]^-{\widetilde{\Lambda}}_-{\cong} & 
\mathbb{R}^{n}/\Lambda(\mathbb{Z}^{n})
\cong \prod_{i=1}^{n}S_{i}
 \ar@{->>}[r]^-{\iota} & 
T^{n}
}
\end{align}
where $\mathbb{T}^{n}$ is the standard $n$-dimensional torus in $\mathbb{C}^{n}$ and the surjective homomorphism $\iota$ is induced from the product of the injective homomorphisms $\iota_{i}:S_{i}\to T^{n}$. 
Put 
\begin{align*}
\widetilde{T}^{n}:=\prod_{i=1}^{n}S_{i},
\end{align*}
then via the isomorphism $\widetilde{\Lambda}$ we may regard $\widetilde{T}^{n}$ as the standard torus $\mathbb{T}^{n}$.
By this identification, the standard $\widetilde{T}^{n}$-action on the unit $2n$-dimensional sphere $S(\oplus_{i=1}^{n}\mathbb{C}\lambda_{i}\oplus \mathbb{R})$ may be regarded as the standard $\mathbb{T}^{n}$-action on $\mathbb{S}^{2n}:=S(\mathbb{C}^{n}\oplus\mathbb{R})$, where the symbol $S(V(\rho)\oplus \mathbb{R})$ for a complex $n$-dimensional $T^{n}$-representation space $V(\rho)$ represents the unit sphere (with respect to the standard metric) in $\mathbb{C}^{n}\times \mathbb{R}\simeq V(\rho)\oplus \mathbb{R}$ with the torus action induced from the representation $\rho:T^{n}\to T^{n}$.

The sphere $\mathbb{S}^{2n}$ with the standard $\mathbb{T}^{n}$-action is known as a torus manifolds with two fixed points, 
and its characteristic submanifolds $M_{i}\subset \mathbb{S}^{2n}$, $i=1,\ldots, n$, are defined by 
\begin{align}
\label{ch_mfd}
M_{i}=\{(z_{1},\ldots,z_{n},r)\in \mathbb{S}^{2n}\subset \mathbb{C}^{n}\oplus \mathbb{R}\ |\ z_{i}=0\}.
\end{align}
Therefore, its normal orientation can be canonically determined by the orientation of the $i$th complex space.
Denote
\begin{align*}
G(\Lambda):=\ker \iota\circ\widetilde{\Lambda}\subset \mathbb{T}^{n}.
\end{align*}
We also denote the complex one-dimensional representation $\iota_{i}\circ p_{i}\circ \widetilde{\Lambda}$ by
\begin{align}
\label{one-dim-rep}
\mu_{i}:=\iota_{i}\circ p_{i}\circ \widetilde{\Lambda}:\mathbb{T}^{n}\to \widetilde{T}^{n}\to S_{i}\hookrightarrow T^{n},
\end{align} 
where $p_{i}$ is the $i$th projection of $\widetilde{T}^{n}$.
The following group is a product of cyclic groups
because it is a finite subgroup of an abelian group (see Remark~\ref{factor}): 
\begin{align*}
G(\Lambda)=\bigcap_{i=1}^{n}\ker \mu_{i}\cong \mathbb{Z}^{n}/\Lambda(\mathbb{Z}^{n}).
\end{align*}
Then, we can consider the following action induced from the standard action $(\mathbb{S}^{2n}, \mathbb{T}^{n})$:
\begin{align*}
(\mathbb{S}^{2n}/G(\Lambda), \mathbb{T}^{n}/G(\Lambda))
\simeq (\mathbb{S}^{2n}/G(\Lambda), T^{n}).
\end{align*} 
It is easy to check that there are exactly two fixed points of this action and that the characteristic suborbifolds $X_{i}\subset \mathbb{S}^{2n}/G(\Lambda)$, $i=1,\ldots, n$, are given by 
\begin{align*}
X_{i}=M_{i}/G(\Lambda),
\end{align*}
where the $M_{i}$'s are defined in equation \eqref{ch_mfd}.
Since only the $i$th factor of $\mathbb{T}^{n}$ fixes $M_{i}$, 
$X_{i}$ is fixed by $\mu_{i}(\mathbb{T}^{n})=S_{i}={\rm exp}\mathbb{R}\lambda_{i}\subset T^{n}$.
Here, we define the normal orientation of $X_{i}$ as the normal orientation of $M_{i}$. 
This shows that the characteristic function of $(\mathbb{S}^{2n}/G(\Lambda), \mathbb{T}^{n}/G(\Lambda))$ coincides with that of $(X(\Lambda), T^{n})$.
Hence, by using Lemma~\ref{fundamental-lem}, we have the following theorem:
\begin{theorem}
\label{top-type}
Assume $n\ge 2$.
Let $X(\Lambda)$ be a torus orbifold over $\Sigma \Delta^{n-1}$ with a characteristic function $\Lambda$.
Then $(X(\Lambda),T^{n})$ is equivariantly homeomorphic to $(\mathbb{S}^{2n}/G(\Lambda), \mathbb{T}^{n}/G(\Lambda))$.
\end{theorem}

\begin{remark}
If $G(\Lambda)$ is the trivial group, then $(\mathbb{S}^{2n}/G(\Lambda), \mathbb{T}^{n}/G(\Lambda)))$ is the torus manifold $(\mathbb{S}^{2n}, \mathbb{T}^{n})$.
This fact is generalized in Proposition~\ref{final-prop}.
\end{remark}

\begin{remark}
\label{factor}
By using the Smith normal form, 
there exist $P, Q\in GL(n;\mathbb{Z})$ such that 
the sequence \eqref{sequence} can be written as follows:
\begin{align*}
\begin{split}
\xymatrix{
\mathbb{T}^{n}=\mathbb{R}^{n}/\mathbb{Z}^{n}\ar[d]_{P} \ar[r]^-{\widetilde{\Lambda}}_-{\cong} & 
\mathbb{R}^{n}/\Lambda(\mathbb{Z}^{n})
\ar[d]_{Q}
 \ar@{->>}[r]^-{\iota} & 
T^{n}\ar[d]_{\cong} \\
\mathbb{R}^{n}/P(\mathbb{Z}^{n}) \ar[r]^-{\widetilde{\Lambda'}}_-{\cong} & 
\mathbb{R}^{n}/(r_{1}\mathbb{Z}\oplus \cdots \oplus r_{n}\mathbb{Z})
 \ar@{->>}[r]^-{\iota} & 
\mathbb{T}^{n}/(C_{r_{1}}\times \cdots \times C_{r_{n}})
}
\end{split}
\end{align*}
where 
\begin{align*}
\Lambda'=Q\Lambda P^{-1}=
\begin{pmatrix}
    r_{1} & & \\
    & \ddots & \\
    & & r_{n}
  \end{pmatrix}
\end{align*}
for some positive integers $r_{1},\ldots, r_{n}$ such $r_{1} | r_{2} |\cdots | r_{n}$, and $C_{r}\cong \mathbb{Z}/r\mathbb{Z}$ is the cyclic subgroup in $\mathbb{T}^{1}$.
Here, we can compute $r_{i}$, $i=1,\ldots, n$, as 
\begin{align*}
r_{i}=\frac{m_{i}(\Lambda)}{m_{i-1}(\Lambda)},
\end{align*}
where $m_{0}(\Lambda):=1$ and $m_{i}(\Lambda)$ is the {\it $i$\textup{th} determinant divisor}, i.e., 
the greatest common divisor of all $i\times i$ minors of $\Lambda$.
\end{remark}

\begin{remark}
One of the consequences of Theorem~\ref{top-type} is
that $X(\Lambda)$ with $n\ge 2$ is equivariantly homeomorphic to a {\it global quotient}, i.e., $X(\Lambda)$ is obtained by the quotient of a finite group action on a manifold; therefore, $X(\Lambda)$ is a {\it good orbifold} (see \cite{ALR}).
On the other hand, we can easily check that a $2$-dimensional torus orbifold with two fixed points is a spindle $S^{2}(m,n)$ as described in Example~\ref{spindle}.
It is well-known that the spindle $S^{2}(m,n)$ for $|m|\not=|n|$ is a {\it bad orbifold}, i.e., it can not be obtained by the global quotient of $S^{2}$.
However, we note that every torus orbifold $S^{2}(m,n)$ is equivariantly homeomorphic to the torus manifold $S^{2}$ as a topological space (ignoring the orbifold structure).
\end{remark}

Theorem~\ref{top-type} leads us to the following corollary which also can be obtained by applying \cite[Proposition 2.12]{GKRW} to the case of a torus orbifold over $Q$:
\begin{corollary}
Let $U_{x}=X(\mathbb{R}^{n}_{\geq},\lambda|_{\mathbb{R}^{n}_{\geq}})$ be the open invariant neighborhood around a fixed point $x$ of a torus orbifold $X(Q,\lambda)$, i.e., 
$\lambda|_{\mathbb{R}^{n}_{\geq}}$ is a characteristic function on 
$\mathbb{R}^{n}_{\geq}$ obtained by restricting $\lambda$ to facets around $x$.
Let $\Lambda$ be the $(n\times n)$-matrix as in \eqref{matrix} which defines the characteristic function $\lambda|_{\mathbb{R}^{n}_{\geq}}$.
Then, the  following holds:
\begin{align*}
(U_{x},T^{n})\simeq (\mathbb{C}^{n}/G(\Lambda), \mathbb{T}^{n}/G(\Lambda)).
\end{align*}
Furthermore, there is the following special orbifold chart around $x$: 
\begin{align*}
(U_{x},V_{x}, H_{x}, p_{x})=(U_{x}, \mathbb{C}^{n}, G(\Lambda), p_{x}:\mathbb{C}^{n}\to \mathbb{C}^{n}/G(\Lambda)\simeq U_{x}).
\end{align*}
\end{corollary}

Note that $G(\Lambda)$ acts on $S^{2n-1}=\{(z_{1},\ldots, z_{n},0)\in S(\mathbb{C}^{n}\oplus \mathbb{R})\}$.
Denote its orbit space by
\begin{align*}
L(\Lambda):=S^{2n-1}/G(\Lambda),
\end{align*} 
which is called an {\it orbifold lens space} in \cite{BSS}.
We note that this orbifold $L(\Lambda)$ has a natural $T^{n}$-action.
Moreover, $U_{x}$ is a cone on $L(\lambda)$ which gives the following corollary:
\begin{corollary}
\label{suspension-lens-sp}
The torus orbifold $(X(\Lambda), T^{n})$ for $n \ge 2$ is equivariantly homeomorphic to $(\Sigma L(\Lambda),T^{n})$, where the $T^{n}$-action on the suspension $\Sigma L(\Lambda)$ is the natural extension of the $T^{n}$-action on $L(\Lambda)$.
\end{corollary}

\begin{remark}
When $G(\Lambda)$ is isomorphic to a cyclic group $C_{p}$ and acts freely on $S^{2n-1}$, then $L(\Lambda)$ is a lens space.
Kawasaki \cite{Ka} considers the case when a cyclic group $C_{p}$ acts almost freely on $S^{2n-1}$ and calls the quotient space 
$S^{2n-1}/C_{p}$ the {\it twisted lens space} which is an orbifold in general (in \cite{Al, BFR}, a twisted lens space is also called a {\it weighted lens space}). 
\end{remark}

It is well-known that 
the $T^{1}$-action on $S^{2}(m,n)$ is equivariantly homeomorphic to the standard $\mathbb{T}^{1}$-action on $\mathbb{S}^{2}$.
Therefore, for any $m,n(\not=0)$, their equivariant cohomologies are isomorphic.
Finally we prove the following theorem which generalizes the fact $S^{2}(m,m)=\mathbb{S}^{2}/C_{m}$.
\begin{theorem}
\label{final-prop}
If $\Lambda$ is a diagonal matrix, then the torus orbifold $X(\Lambda)$ is equivariantly homeomorphic to the torus manifold obtained from the identity matrix, i.e., 
the standard $2n$-dimensional sphere $\mathbb{S}^{2n}$ with $\mathbb{T}^{n}$-action.
\end{theorem}
\begin{proof}
If $\Lambda$ is the diagonal matrix $\Lambda'$ in Remark~\ref{factor} then $X(\Lambda)$ is the orbifold $S^{2n}/C_{r_{1}}\times \cdots\times C_{r_{n}}$, where $C_{r_{1}}\times \cdots \times C_{r_{n}}$ acts on the complex coordinates in $S^{2n}\subset \mathbb{C}^{n}\oplus \mathbb{R}$ coordinatewise.
Then, because the scalar products on the characteristic functions does not change the topological type of the underlying topological space, this is equivariantly homeomorphic to the torus manifold obtained by taking $\Lambda=I_{n}$.
\end{proof}

\section{The equivariant cohomology of $X(\Lambda)$ with $H^{odd}(X(\Lambda))=0$}
\label{sect:5}

In this final section, we compute the equivariant cohomology of $X(\Lambda)$ with $H^{odd}(X(\Lambda))=0$ by using the formula in \cite{DKS}.
In what follows, we often identify $\mathfrak{t}_{\mathbb{Z}}^{*}:={\rm Hom}(\mathfrak{t}_{\mathbb{Z}},\mathbb{Z})$ with $H^{2}(BT;\mathbb{Z})\cong \mathbb{Z}^{n}$.

We first consider the ordinary cohomology of $X(\Lambda)$. 
As $X(\Lambda)$ is simply connected (because of Corollary~\ref{suspension-lens-sp}), $H^{1}(X(\Lambda))=0$. Moreover, for $H^{3}(X(\Lambda))$, the following lemma holds:
\begin{lemma}
\label{3rd-degree}
Assume $n\ge 2$. 
Let $N$ be the smallest subgroup of $G(\Lambda)$ which contains 
all those elements of $G(\Lambda)$ that fix points in $S^{2n-1}$. 
Then
$H^{3}(X(\Lambda))\cong G(\Lambda)/N$.
\end{lemma}
\begin{proof}
Since $n\ge 2$, $S^{2n-1}$ is simply connected.
Moreover, $G(\Lambda)$ acts on $S^{2n-1}$ effectively.
Therefore, by \cite{A},
we have that $\pi_{1}(S^{2n-1}/G(\Lambda))\cong G(\Lambda)/N\cong H_{1}(L(\Lambda))$.
This shows that $H_{1}(L(\Lambda))$ is torsion if $G(\Lambda)/N$ is not the identity group; namely, $G(\Lambda)\not=N$.
By using the universal coefficient theorem, we have 
$H_{1}(L(\Lambda))\cong H^{2}(L(\Lambda))$.
Therefore, it follows from the Mayer--Vietoris exact sequence that 
$H^{2}(L(\Lambda))\cong H^{3}(\Sigma L(\Lambda))=H^{3}(X(\Lambda))$.
\end{proof}
In particular, if $(S^{2n-1})^{G(\Lambda)}\not=\emptyset$, then $H^{3}(X(\Lambda))=0$.
Moreover, $H^{odd}(X(\Lambda))=0$ if  
$G(\Lambda)=\{e\}$.
Namely, if $\det \Lambda=\pm 1$, then $H^{odd}(X(\Lambda))=0$.

\begin{remark}
\label{rem-prod}
Suppose that the product of cyclic groups
$G=C_{r_{1}}\times\cdots\times C_{r_{n}}$ acts on $S^{2n-1}\subset \mathbb{C}^{n}$ coordinatewise (also see Remark~\ref{factor}).
Let $N$ be the smallest subgroup in $G$ which contains 
all those elements of $G$ which have fixed points.
Then, $N=G$.
Therefore, by Lemma~\ref{3rd-degree}, $H^{3}(S^{2n}/G)=0$.
\end{remark}

\subsection{Orbifold torus graph of $(Q,\Lambda)$}
\label{sect:5.1}

Let $(Q, \Lambda)$ be a pair of a manifold with two vertices and a characteristic matrix.
Note that if $\dim Q=1$ then $(Q, \Lambda)$ is always $([-1,1],\lambda)$ induced from the spindle $S^{2}(m,n)$ discussed in Example~\ref{spindle} and if $\dim Q\ge 2$ then we may define the characteristic function as the matrix given by \eqref{matrix} because the combinatorial structure is the same as $\Sigma \Delta^{n-1}$. 

We shall define an {\it orbifold torus graph} $(\Gamma,\alpha)$ of $(Q,\Lambda)$ as follows. 
Let $\Gamma=(\mathcal{V}(\Gamma), \mathcal{E}(\Gamma))$ be the one-skeleton of $Q$, i.e., 
$\mathcal{V}(\Gamma)$ is the set of the two vertices and $\mathcal{E}(\Gamma)$ is  the set of the $n$ edges of $Q$.
By the combinatorial structure of $Q$, 
we may regard the edge $e_{j}\in \mathcal{E}(\Gamma)$, $j=1,\ldots, n$, as a connected component of the intersection of exactly $(n-1)$ facets, i.e., $\mathfrak{F}(Q)\setminus \{F_{j}\}$; and we also denote the normal facet of $e_{j}$ by $F_{j}$, i.e., $F_{j}\cap e_{j}=\mathcal{V}(Q)$.
Now, we define a function 
\begin{align*}
\alpha:\mathcal{E}(\Gamma) \to \mathfrak{t}^{*}_\mathbb{Q}
\end{align*}
by the following system of equations: 
\begin{equation}
\label{eq_how_to_determine_axial_ftn}
\left< \alpha(e_{j}), \lambda(F_{i}) \right>=
\begin{cases} 
0\ \text{if $i\not=j$};\\
1\ \text{if $i=j$},  
\end{cases}
\end{equation}
where $\langle~,~ \rangle$ denotes the natural paring between
a vector space $\mathfrak{t}_{\mathbb{Q}}$ and its dual space $\mathfrak{t}_{\mathbb{Q}}^{*}$.
Note that $\lambda(F_{i})\in \mathfrak{t}_{\mathbb{Z}}\subset \mathfrak{t}_{\mathbb{Q}}$. 
We call such a labeled graph $(\Gamma,\alpha)$ an {\it orbifold torus graph} of $(Q,\Lambda)$.

\begin{remark}
In \cite{DKS}, we define an {\it abstract orbifold torus graph} which is a generalization of torus graphs in \cite{MMP}.
\end{remark}

Denote by $\overline{e}$ the reversed oriented edge of $e$, i.e.,
$i(e)=t(\overline{e})$ and $t(e)=i(\overline{e})$.

\begin{example}
\label{spindle_GKM}
The spindle $S^{2}(m,n)$ is defined by $([-1,1], \lambda)$ in Example~\ref{spindle}.
Set the generator of $\mathfrak{t}_{\mathbb{Z}}^{*}$ by $x$, i.e., $\mathfrak{t}_{\mathbb{Z}}^{*}=\mathbb{Z}x$ and $\mathfrak{t}_{\mathbb{Q}}^{*}=\mathbb{Q}x$. 
In this case, $[-1,1]$ itself, say $e$, is the edge in this manifold with faces and the two vertices $i(e)=\{-1\}$ and $t(e)=\{1\}$ are the only facets.
It follows from the definition of an orbifold torus graph of $([-1,1], \lambda)$ that we have the following axial function (see Figure~\ref{spindle_GKM_fig}):
\begin{eqnarray*}
\alpha(e)=\frac{1}{m}x,\quad \alpha(\overline{e})=\frac{1}{n}x.
\end{eqnarray*}
\begin{figure}[h]
\begin{tikzpicture}
\draw (0,0)--(3,0);
\node [below] at (0,-0.1) {$-1$};
\fill (0,0) circle [radius=2.5pt];
\node [below] at (3,-0.1) {$1$};
\fill (3,0) circle [radius=2.5pt];
\draw[very thick, ->] (0,0)--(0.8,0);
\node[above] at (0.8,0.1) {$\frac{1}{m}x$};
\draw[very thick, ->] (3,0)--(2.2,0);
\node[above] at (2.2,0.1) {$\frac{1}{n}x$};
\end{tikzpicture}
\caption{The orbifold torus graph of $([-1,1], \lambda)$.}
\label{spindle_GKM_fig}
\end{figure}
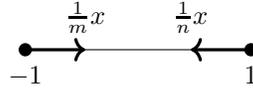
\end{example}

\begin{example}
\label{ex-orbsphere}
Assume $\dim Q \ge 2$.
Let $\Lambda^{*}$ be the transpose of the cofactor matrix of $\Lambda$, i.e., 
$\Lambda^{*}\Lambda=(\det \Lambda) I_{n}$.
We may put
\begin{align}
\label{cofactor_matrix}
\Lambda^{*}=
\begin{pmatrix}
\mu_{1} \\
\vdots \\
\mu_{n}
\end{pmatrix},
\end{align}
where $\mu_{i}\in \mathbb{Z}^{n}\cong \mathfrak{t}_{\mathbb{Z}}^{*}$ is the one-dimensional representation of $\mathbb{T}^{n}$ defined in equation \eqref{one-dim-rep}.
Then, by equation \eqref{eq_how_to_determine_axial_ftn}, the vector $\mu_{i}\in \mathfrak{t}_{\mathbb{Z}}^{*}$ satisfies the following equation:
\begin{align}
\label{axial-fct-2-vertices}
\alpha(e_{i})=\frac{1}{\det \Lambda} \mu_{i},
\end{align}
where $e_{i}$ is the edge of $\Sigma\Delta^{n-1}$ which is not contained in the facet $F_{i}$ (see Figure~\ref{figure}).
We put $D:=\det \Lambda$.
\begin{figure}[h]
\begin{tikzpicture}
\begin{scope}[xscale=1.3]
\draw[fill=yellow] (0,0) circle [radius=1];
\node at (-1,0) {$\bullet$};
\node at (1,0) {$\bullet$};
\node[above] at (0,1) {$(1,5)$};
\node[below] at (0,-1) {$(1,3)$};
\end{scope}

\begin{scope}[xshift=170, xscale=1.3]
\draw (0,0) circle [radius=1];
\node at (-1,0) {$\bullet$};
\node at (1,0) {$\bullet$};

\draw [->, thick] (-1,0) arc (180:150:1);
\draw [->, thick] (-1,0) arc (180:210:1);
\draw [->, thick] (1,0) arc (0:30:1);
\draw [->, thick] (1,0) arc (0:-30:1);

\node[left] at (-1.5/2, 1.7/2) {$\frac{5}{2}x-\frac{1}{2}y$};
\node[right] at (1.5/2, 1.7/2) {$\frac{5}{2}x-\frac{1}{2}y$};

\node[left] at (-1.5/2, -1.7/2) {$-\frac{3}{2}x+\frac{1}{2}y$};
\node[right] at (1.5/2, -1.7/2) {$-\frac{3}{2}x+\frac{1}{2}y$};

\end{scope}
\end{tikzpicture}

\caption{The case when $n=2$. The left figure is $(\Sigma \Delta^{1}, \lambda)$ and the right one is its orbifold torus graph, where $\mu_{1}=5x-y$ and $\mu_{2}=-3x+y$ for some generators $x,y$ in $\mathfrak{t}^{*}_{\mathbb{Z}}$.}
\label{figure}
\end{figure}
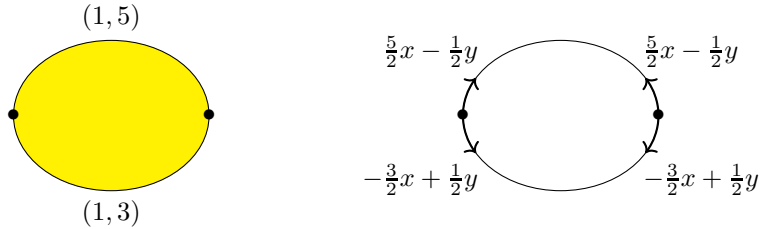
\end{example}

For an orbifold torus graph $(\Gamma,\alpha)$, we define the following rings.
\begin{definition}[Graph equivariant cohomology]
\label{eq-GKM}
The following ring is said to be the (integral) {\it graph equivariant cohomology ring}:
\begin{equation*}
\label{eq_GKM-graph_cohom}
H_{T}^{*}(\Gamma, \alpha)=
\left\{f:\mathcal{V}(\Gamma) \to 
H^{*}(BT;\mathbb{Z}) \mid 
f(i(e))\equiv f(t(e)) \text{ mod } r_e \alpha(e) \right\},
\end{equation*}
where $r_{e}$ is the minimal positive integer such that $r_{e}\alpha(e)\in \mathfrak{t}^{*}_\mathbb{Z}\cong H^{2}(BT;\mathbb{Z})$ and 
$i(e)$ (resp.~$t(e)$) is the initial (resp.~terminal) vertex of an oriented edge $e$.
\end{definition}
Here, $H_{T}^{*}(\Gamma, \alpha)$ may be regarded as 
an $H^{*}(BT;\mathbb{Z})$-subalgebra of $\bigoplus_{v\in \mathcal{V}(\Gamma)} H^{*}(BT;\mathbb{Z})$. 
Then $r_{e}\alpha(e)\in H^{2}(BT;\mathbb{Z})$ and 
there is a natural grading in $H_{T}^{*}(\Gamma, \alpha)$ induced by the grading of $H^{*}(BT;\mathbb{Z})$.
Note that all edges satisfy $r_{e}\alpha(e)=\pm r_{\overline{e}}\alpha(\overline{e})$ (see \cite{DKS}).
We may define 
the cohomology of $(\Gamma, \alpha)$ over rational coefficients
as follows:
\begin{equation*}
\label{eq_rational_graph_equiv_cohom}
H_{T}^{*}(\Gamma, \alpha;\mathbb{Q})=
\left\{f:\mathcal{V}(\Gamma) \to 
H^{*}(BT;\mathbb{Q})
 \mid f(i(e))\equiv f(t(e)) \text{ mod } \alpha(e) \right\}.
\end{equation*}
Similarly, this has the natural 
$H^{*}(BT;\mathbb{Q})$-algerbra structure. 
This coincides with the definition of the cohomology ring of a GKM graph in \cite[Section 1.7]{GZ}.
One can see that 
$H_{T}^{*}(\Gamma,\alpha)$ is a subring of 
$H_T^{*}(\Gamma, \alpha;\mathbb{Q})$.
We call $H_{T}^{*}(\Gamma, \alpha;\mathbb{Q})$ the {\it rational graph equivariant cohomology}.

The next theorem is a consequence of applying the main result of \cite{DKS} restricted to the case of torus orbifolds with two fixed points.
\begin{theorem}
\label{main1-DKS}
Assume $X:=X(Q,\lambda)$ satisfies $H^{odd}(X)=0$.
Then, there is an isomorphism
\begin{align*}
H^{*}_{T}(X)\cong H_{T}^{*}(\Gamma,\alpha),
\end{align*}
where $(\Gamma,\alpha)$ is the orbifold torus graph of $(Q,\lambda)$.
\end{theorem}

\begin{remark}
The above theorem also holds for more general GKM orbifolds satisfying certain conditions (see \cite{DKS}).
\end{remark}

\subsection{Weighted face ring}
\label{sect:5.2}
Given an $n$-valent orbifold torus graph $(\Gamma, \alpha)$ of $(Q,\lambda)$. 
Let $F$ be an $(n-k)$-dimensional face in $Q$ (see Section \ref{sect:2}).
Then
there is an $(n-k)$-valent subgraph $\Gamma_{F}$ which is the one-skeleton of $F$.
We call this subgraph an {\it $(n-k)$-dimensional face} of $(\Gamma,\alpha)$. 
Each face $\Gamma_{F}$ defines a {\it rational Thom class} 
$\tau_F \in H^{2k}_{T}(\Gamma, \alpha;\mathbb{Q})$ as follows: 
\begin{equation*}
\label{eq_rational_thom_class}
\tau_F(v):=
\begin{cases}
\prod_{\substack{i(e)=v \\ e\notin \mathcal{E}(\Gamma_{F})}  }
\alpha(e) & \text{if } v\in \mathcal{V}(\Gamma_{F}); \\
0 & \text{otherwise}. 
\end{cases}
\end{equation*}
Note that $\deg \tau_{F}=2k$ for a codimension-k face $F$.
We formally define
$\tau_{\emptyset}=0, \ \tau_{\Gamma}=1.$

\begin{example}
The following two figures (Figure~\ref{Thom_class1} and Figure~\ref{Thom_class2}) are the examples of rational Thom classes of the orbifold torus graph in Figure~\ref{figure}.
\begin{figure}[h]
\begin{tikzpicture}[xscale=1.3]
\draw (-1,0) arc (180:0:1);
\draw[dashed] (-1,0) arc (180:360:1);
\node at (-1,0) {$\bullet$};
\node at (1,0) {$\bullet$};
\node[above] at (0,1) {$F$};
\node[above left] at (-1,0) {$p$};
\node[above right] at (1,0) {$q$};
\node[below left] at (-1,0) {$-\frac{3}{2}x+\frac{1}{2}y$};
\node[below right] at (1,0) {$-\frac{3}{2}x+\frac{1}{2}y$};
\end{tikzpicture}
\caption{The rational Thom class of the facet $F$ in Figure~\ref{figure}, i.e., $\tau_{F}(p)=\tau_{F}(q)=-\frac{3}{2}x+\frac{1}{2}y$.}
\label{Thom_class1}
\end{figure}
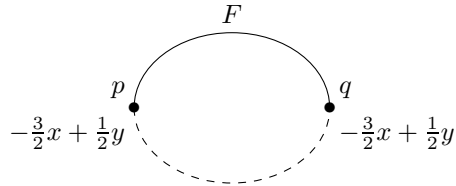

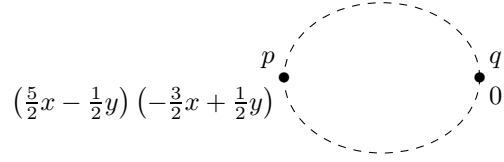
\begin{figure}[h]
\begin{tikzpicture}[xscale=1.3]
\draw[dashed] (0,0) circle [radius=1];
\node at (-1,0) {$\bullet$};
\node at (1,0) {$\bullet$};
\node[above left] at (-1,0) {$p$};
\node[above right] at (1,0) {$q$};
\node[below left] at (-1,0) {$\left(\frac{5}{2}x-\frac{1}{2}y\right)\left(-\frac{3}{2}x+\frac{1}{2}y\right)$};
\node[below right, white] at (1,0) {$\left(\frac{5}{2}x-\frac{1}{2}y\right)\left(-\frac{3}{2}x+\frac{1}{2}y\right)$};
\node[below right] at (1,0) {$0$};
\end{tikzpicture}
\caption{The rational Thom class of the vertex $p$ in Figure~\ref{figure}, i.e., $\tau_{p}(p)=(\frac{5}{2}x-\frac{1}{2}y)(-\frac{3}{2}x+\frac{1}{2}y)$ and $\tau_{p}(q)=0$.}
\label{Thom_class2}
\end{figure}
\end{example}

Let $\mc{F}$ be the set of all faces in $(\Gamma,\alpha)$
and 
$\Z[\tau_F \mid F\in \mc{F}]$ be the graded polynomial ring generated by the rational Thom classes, where the grading is given by the degree of $\tau_{F}$. 
Set the graded subring $\mathscr{Z}_{\Gamma,\alpha}$ of $\mathbb{Z}[\tau_{F}\ |\ F\in \mathcal{F}]$ as 
\begin{align*}
\mathscr{Z}_{\Gamma,\alpha} :=
\left\{ f\in \Z[\tau_F \mid F\in \mc{F}] ~\big|~ \forall v\in \mathcal{V}(\Gamma),~
f(v) \in  H^\ast(BT^n;\Z) \right\}.
\end{align*}
Then, it is easy to see that the elements in $\mathscr{Z}_{\Gamma,\alpha}$ 
of the form 
\begin{align}
\label{relation}
\tau_{E}\tau_{F}-\tau_{E\vee F}\sum_{G\in E\cap F}\tau_{G}
\end{align}
are $0$ for all vertices $v\in \mathcal{V}(\Gamma)$, 
where 
the summation runs through all connected components in $E\cap F$. Here the symbol $E\vee F$ represents the minimal face which contains both $E$ and $F$; note that if $E\cap F\not=\emptyset$ then the face $E\vee F$ can be uniquely determined (also see \cite{MMP}).
Therefore, we can define the ideal $\mathcal{I}$ of $\mathscr{Z}_{\Gamma,\alpha}$ generated by all elements defined by \eqref{relation}.
Set 
\begin{align*}
\mathbb{Z}[\Gamma,\alpha]:=
\mathscr{Z}_{\Gamma,\alpha}/\mathcal{I}.
\end{align*}
We call this ring the {\it weighted face ring} of $(\Gamma,\alpha)$.

The following theorem is one of the main results in \cite[Theorem 3.6]{DKS}:
\begin{theorem}
Let $(\Gamma,\alpha)$ be the orbifold torus graph induced from $(Q,\lambda)$.
Then the following graded rings are isomorphic:
\begin{align*}
H^{*}_{T}(\Gamma,\alpha)\cong \mathbb{Z}[\Gamma,\alpha].
\end{align*} 
\end{theorem}
Therefore, together with Theorem~\ref{main1-DKS}, we have the following corollary:
\begin{corollary}
\label{main3-DKS}
If the torus orbifold $X(Q,\lambda)$ satisfies $H^{odd}(X(Q,\lambda))=0$, then its equivariant cohomology satisfies 
\begin{align*}
H^{*}_{T}(X(Q,\lambda))\cong \mathbb{Z}[\Gamma,\alpha],
\end{align*} 
where $(\Gamma,\alpha)$ is the induced orbifold GKM graph from $(Q,\lambda)$.
\end{corollary}

\subsection{The equivariant cohomology of $X(\Lambda)$ when  $H^{odd}(X(\Lambda))=0$}
\label{sect:5.3}

Now we may compute the equivariant cohomology of $X(\Lambda)$ when $H^{odd}(X(\Lambda))=0$.
If $n=1$, we may think $X(\Lambda)$ as a spindle.
Due to Section~\ref{sect:2}, 
the graph $\Gamma=(\mathcal{V}(\Gamma),\mathcal{E}(\Gamma))$ of $Q$ (manifold with two vertices) is given by 
\begin{align*}
\mathcal{V}(\Gamma)=\{p,\ q\},\quad \mathcal{E}(\Gamma)=\{e_{1},\ldots, e_{n}, \overline{e_{1}},\ldots, \overline{e_{n}}\}
\end{align*}
where $i(e_{i})=p$ and $t(e_{i})=q$ for $i=1,\ldots, n$,
and the axial function $\alpha:E(\Gamma)\to H^{2}(BT;\mathbb{Q})$ is given by equation \eqref{axial-fct-2-vertices}, i.e., 
\begin{align}
\label{axial-5.3}
\alpha(e_{i})=\alpha(\overline{e_{i}})=\frac{1}{D} \mu_{i}.
\end{align}
Then the rational Thom class $\tau_{i}:=\tau_{F_{i}}$ corresponding to a facet $F_{i}$ (the normal facet of $e_{i}$), $i=1,\ldots, n$, is an element in $\mathscr{Z}_{\Gamma,\alpha}$.
Indeed, 
\begin{align}
\label{Thom-5.3}
\tau_{i}(p)=\tau_{i}(q)=\frac{1}{D} \mu_{i}\in H^{2}(BT;\mathbb{Q}).
\end{align}
Set $\tau_{p}$ and $\tau_{q}$ as the Thom classes of the vertices, i.e., 
\begin{align}
\label{Thom-pt-5.3}
\tau_{p}(p)=\frac{1}{D^n} \prod_{i=1}^{n}\mu_{i}=\tau_q(q),\quad \tau_{p}(q)=0=\tau_q(p).
\end{align}
Let $a_{i}$ ($i=1,\ldots, n$), $a_{p}$ and $a_{q}$ be the smallest positive integers such that $a_{i}\tau_{i}, a_{p}\tau_{p}, a_{q}\tau_{q}\in H^{*}_{T}(\Gamma,\alpha)$. 
Note that 
$a_{p}|D$ and $a_{q}|D$ by equations \eqref{axial-5.3} and \eqref{Thom-pt-5.3}.
We put
\begin{align*}
\widetilde{\tau}_{i}:=a_{i}\tau_{i}\ (i=1,\ldots, n),\quad 
\widetilde{\tau}_{p}:=a_{p}\tau_{p}\quad {\rm and}\quad \widetilde{\tau}_{q}:=a_{q}\tau_{q}.
\end{align*}
Define the matrix $\Lambda^{*}_{/\ell}$ (which modifies the matrix $\Lambda^{*}$ defined in \eqref{cofactor_matrix}) by
\begin{align*}
\begin{pmatrix}
\mu_{1}/\ell_{1} \\
\vdots \\
\mu_{n}/\ell_{n}
\end{pmatrix}
=
\begin{pmatrix}
\mu_{11} & \cdots & \mu_{1n} \\
\vdots & \ddots & \vdots \\
\mu_{n1} & \cdots & \mu_{nn}
\end{pmatrix}
\end{align*}
where $\ell_{i}$ is the greatest common divisor of the entires of the row vector $\mu_{i}$.
It is easy to see that $\ell_{i}=D/a_{i}$.
Then we have the following diagonal matrix
\begin{align}
\label{key-matrix}
\Lambda^{*}_{/\ell} \Lambda
=\begin{pmatrix}
    a_{1} & & \\
    & \ddots & \\
    & & a_{n}
  \end{pmatrix}.
\end{align}

\begin{theorem}
Let $(\Gamma,\alpha)$ be the orbifold torus graph induced from an orbifold characteristic pair $(\Sigma\Delta^{n-1},\lambda)$. Then
\begin{align*}
\mathbb{Z}[\Gamma,\alpha]\cong \mathbb{Z}[x_{1},\ \ldots,\ x_{n},\ \widetilde{\tau}_{p},\ \widetilde{\tau}_{q}]/\langle \mu_{1}(x)\cdots\mu_{n}(x)-(\widetilde{\tau}_{p}+\widetilde{\tau}_{q}),\ \widetilde{\tau}_{p}\widetilde{\tau}_{q} \rangle,
\end{align*}
where $x_{i}=\lambda_{i1}\tau_1+\dots+\lambda_{in}\tau_n$ $($see \eqref{matrix}$)$ and 
$\mu_{i}(x)=\mu_{i1}x_1+\dots+\mu_{in}x_n$.
\end{theorem}

\begin{proof}
 The Thom classes corresponding to the facets $\tau_i=\tau_{F_i}$ are given by the global polynomials defined in equation \eqref{Thom-5.3}.
Recall from \cite[Remark 4.9 (1)]{DKS} the elements $x_{1},\ldots, x_{n}$ are the elements in $\mathbb{Z}[\Gamma,\alpha]$.
In particular, they form a basis in $H^{2}(BT;\mathbb{Z})$,
by which any global element in $\mathbb{Z}[\Gamma,\alpha]$ is generated.

Since any $g\in \Z[\Gamma,\alpha]$ with $\deg g<2n$ is
 a global polynomial, it is generated by $x_1,\dots,x_n$. 
If $f\in H^{*}_{T}(\Gamma,\alpha)$ has a degree $\ge 2n$, then there is the integral global polynomial $g\in \Z[x_1,\dots,x_n]$ such that $f(q)-g(q)=0$. Put $h=f-g$. Then it follows from Definition~\ref{eq-GKM} that $h(p)$ is divisible by $a_{p}\tau_{p}(p)(=:\widetilde{\tau}_{p})$. Therefore, there exists a global element $g'\in \Z[x_1,\dots,x_n]$ such that $h=g'\widetilde{\tau}_{p}$. This shows that any $f\in H^{*}_{T}(\Gamma,\alpha)$ is generated by $x_{i}$ ($i=1,\ldots, n$), $\widetilde{\tau}_{p}$ and $\widetilde{\tau}_{q}$.

It follows easily from equation \eqref{relation} that all relations can be generated by $\widetilde{\tau}_{p}\widetilde{\tau}_{q}$
 and $\widetilde{\tau}_{1}\cdots \widetilde{\tau}_{n}-(\widetilde{\tau}_{p}+\widetilde{\tau}_{q})$.
On the other hand, equation \eqref{key-matrix} tells us that $\widetilde{\tau}_{i}=\mu_{i}(x)$.
Therefore, we establish the relations.
\end{proof}

Consequently, we have the following corollary:
\begin{corollary}
\label{eq-cohom-orb}
If $H^{odd}(X(\Lambda))=0$, then 
\begin{align*}
H^{*}_{T}(X(\Lambda))\cong  
\mathbb{Z}[\tau_{1},\ \cdots,\ \tau_{n},\ \tau_{p},\ \tau_{q}]/\langle \tau_{1}\cdots\tau_{n}-(\tau_{p}+\tau_{q}),\ \tau_{p}\tau_{q} \rangle,
\end{align*}
where $\deg \tau_{i}=2$, $i=1,\ldots,n$, and $\deg \tau_{p}=\deg \tau_{q}=2n$.
\end{corollary}

\begin{remark}
If $G(\Lambda)=\{e\}$ then $X(\Lambda)=\mathbb{S}^{2n}$, i.e., a torus manifold.
Therefore, we 
can also obtain the above fact for the case of the standard $2n$-dimensional sphere by using the main theorem of \cite{MMP}.
\end{remark}

\subsection{The equivariant cohomology of $S^{2}(m,n)$.}
\label{sect:5.4}
In this final subsection, we apply Corollary~\ref{eq-cohom-orb} to the case of spindles.

Recall the spindle $S^{2}(m,n)$.
This is homeomorphic to $S^{2}$, therefore, $H^{odd}(S^{2}(m,n))=0$.
The orbifold torus graph of $S^{2}(m,n)$ is the one defined in Figure~\ref{spindle_GKM_fig}.
In this case, the rational Thom classes for the two vertices $p(=-1)$ and $q(=1)$ are defined as follows:
\begin{align*}
\tau_{p}(p)=\frac{1}{m}x,\ \tau_{p}(q)=0 \quad {\rm and}\quad  
\tau_{q}(p)=0,\ \tau_{q}(q)=\frac{1}{n}x,\quad {\rm respectively.}
\end{align*}
Therefore, by applying Corollary~\ref{eq-cohom-orb}, we have:
\begin{corollary}
The $T^{1}$-equivariant cohomology of the spindle $S^{2}(m,n)$ is isomorphic to the following ring:
\begin{align*}
H_{T}^{*}(S^{2}(m,n))
\cong
\mathscr{Z}_{m,n}/\mathcal{I}
\cong
\mathbb{Z}[m\tau_{p}, n\tau_{q}]/\langle mn\tau_{p}\tau_{q} \rangle,
\end{align*}
where $\deg \tau_{p}=\deg \tau_{q}=2$.
\end{corollary}
We note that this is isomorphic to the equivariant cohomology of the standard $T^{1}$-action on $\mathbb{S}^{2}$.


\begin{thebibliography}{99}
\bibitem{ALR}
A.~Adem, J.~Leida and Y.~Ruan,
Orbifolds and Stringy Topology,
Cambridges Tracts in Math.~171, 
Cambridge University Press, 2007.

\bibitem{Al}
A.~AlAmrani,
\textit{Complex K-theory of weighted projective spaces},
J.~of Pure and Appl.~Alg.~\textbf{93} (1994), 113--127.

\bibitem{A}
M.~A.~Armstrong,
\textit{The fundamental group of the orbit space of a discontinuous group},
Proc.~Cambridge Philos.~Soc.~\textbf{64} (1968), 299--301.

\bibitem{BFR2}
A.~Bahri, M.~Franz and N.~Ray, 
{\it The equivariant cohomology ring
of weighted projective space}, 
Math.~Proc.~Cambridge Philos.~Soc.~\textbf{146} (2009), no. 2,
395--405.

\bibitem{BFR}
A.~Bahri, M.~Franz and N.~Ray,
\textit{Weighted projective spaces and iterated Thom spaces},
Osaka J.~Math.~\textbf{51} (2014), 89--121.

\bibitem{BSS}
A.~Bahri, S.~Sarkar and J.~Song, 
\textit{On the integral cohomology ring of toric orbifolds and singular toric varieties},
Algebr.~Geom.~Topol.~\textbf{17} (2017), no. 6, 3779--3810.

\bibitem{BP}
V.~Buchstaber and T.~Panov,
Toric topology. 
Mathematical Surveys and Monographs, 204. American Mathematical Society, Providence, RI, 2015.

\bibitem{DKS}
A.~Darby, S.~Kuroki and J.~Song,
\textit{Equivariant cohomology of torus orbifolds},
arXiv:1809.03678.


\bibitem{D-coxeter}
M.~Davis, 
\textit{When are two Coxeter orbifolds diffeomorphic?}, 
Michigan Math.~J.~\textbf{63} (2014), 401--421.

\bibitem{DJ}
M.~Davis and T.~Januszkiewicz, 
\textit{Convex polytopes, Coxeter orbifolds
and torus actions}, Duke Math.~J.~\textbf{62} (1991), 417--451.

\bibitem{F}
M.~Franz,
\textit{Describing toric varieties and their equivariant cohomology},
Colloq.~Math.~\textbf{121} (2010), 1--16.

\bibitem{GKRW}
F.~Galaz-Garcia, M.~Kerin, M.~Radeschi and M.~Wiemeler
\textit{Torus orbifolds, slice-maximal torus actions and rational ellipticity},
Int.~Math.~Res.~Not.~(2017) 
DOI:10.1093/imrn/rnx064; preprint version:
arXiv:1404.3903.


\bibitem{GZ} 
V.~Guillemin and C.~Zara, 
\textit{One-skeleta, Betti numbers, and equivariant cohomology}, 
Duke Math.~J.~\textbf{107} (2001), 283--349.

\bibitem{HM}
A.~Hattori and M.~Masuda, 
{\it Theory of Multi-fans}, 
Osaka.\ J.\ Math., \textbf{40} (2003), 1--68.

\bibitem{Ka}
T.~Kawasaki,
\textit{Cohomology of twisted projective spaces and lens complexes}, Math.~Ann.~\textbf{206} (1973), 243--248. 


\bibitem{MMP} 
H.~Maeda, M.~Masuda, T.~Panov, 
\textit{Torus graphs and simplicial posets},
Adv.~Math.~\textbf{212} (2007), 458--483.


\bibitem{PS}
M.~Poddar and S.~Sarkar
\emph{On Quasitoric Orbifolds},
Osaka J.~Math.~\textbf{47} (2010), 1055--1076.

\end{thebibliography}
\end{document}